\documentclass[reqno]{amsart}
\usepackage{amsmath}

\begin{document}
\title[Riemann Problem for a limiting system]
{Riemann Problem for a limiting system in elastodynamics}

\author[Anupam Pal Choudhury]
{Anupam Pal Choudhury}

\address{Anupam Pal Choudhury \newline
TIFR Centre for Applicable Mathematics\\
Sharada Nagar, Chikkabommasandra, GKVK P.O.\\
Bangalore 560065, India}
\email{anupam@math.tifrbng.res.in}

\thanks{The author would like to thank the referee for suggesting the corrections and modifications. He would also like to 
avail this opportunity to express his heartiest gratitude to Prof. K.T. Joseph and Prof. Evgeniy Panov for helpful discussions during the preparation of the article}

\begin{abstract}
In this article, we discuss about the resolution of the Riemann problem for a 2x2 system in
nonconservative form exhibiting parabolic degeneracy. The system can be perceived as the limiting
equation (depending on a parameter tending to 0) of a 2x2 strictly hyperbolic, genuinely nonlinear,
non-conservative system arising in context of a model in elastodynamics.
\end{abstract}

\maketitle
\numberwithin{equation}{section}
\numberwithin{equation}{section}
\newtheorem{theorem}{Theorem}[section]
\newtheorem{remark}[theorem]{Remark}

\section{Introduction}
Let us consider the one-dimensional systems in non-conservative form
\begin{equation}
u_{t}+A(u)u_{x}=0,
\notag
\end{equation}
where $u \in \mathbb{R}^n$, $x \in \mathbb{R}$, and the matrix $A$ is smooth. The study of discontinuous solutions of such systems depend crucially on an appropriate
definition of the nonconservative products involved. Many fruitful attempts in this direction have been made in the past (see \cite{a1,c1,c2,d1,r1,v1}). Using
the approach of family of paths introduced in \cite{d1} (we henceforth refer to them as DLM paths), under the conditions of strict hyperbolicity
and genuine non-linearity/linear degeneracy, existence of solutions in the class of functions of bounded variation (BV) was proved in \cite{l1}.
But as soon as one drops the condition of strict hyperbolicity, the situation turns quite different and we can no more expect the solutions
in the same BV class. Rather in many such instances for conservation laws, it has been found that one needs to appeal to the class of singular solutions.
One such deviation from the condition of strict hyperbolicity is that of \textit{parabolic degeneracy} (the case where the matrix $A(u)$ fails to have
a complete set of right eigenvectors). For a class of systems of conservation laws exhibiting parabolic degeneracy, it was shown in \cite{z1} that singular
concentrations tend to occur in one of the unknowns.\\
In this article we aim to study the Riemann problem in the domain $\Omega=\{(x,t):-\infty<x<\infty,\ t>0\}$ for the following non-conservative system exhibiting parabolic degeneracy:
\begin{equation}
\begin{aligned} 
 u_{t}+uu_{x}-\sigma_{x}&=0, \\
 \sigma_{t}+u\sigma_{x}&=0.
 \end{aligned}
\label{e1.1}
\end{equation}
Before proceeding further with our discussion of the properties of the system \eqref{e1.1}, let us briefly consider the following strictly hyperbolic, genuinely nonlinear, non-conservative system:
\begin{equation}
\begin{aligned}
 u_{t}+uu_{x}-\sigma_{x}&=0, \\
 \sigma_{t}+u\sigma_{x}-k^{2}u_{x}&=0.
\end{aligned}
\label{e1.2}
\end{equation}
This system arises in the study of a model in elastodynamics (see \cite{c1,c2}). Here $u$ is the velocity, $\sigma$ is the stress and $k>0$
is an elasticity constant. The initial-value problem for this system in the domain $\Omega=\{(x,t):-\infty<x<\infty,\ t>0\}$ has been well studied (see \cite{c1,c2,j1,j2,j3,l1}). In particular the Riemann problem for this system
was explicitly solved using Volpert's product in \cite{j3}. It would be interesting to note that the system \eqref{e1.1} can be obtained from
the system \eqref{e1.2} by putting $k=0$.\\
Returning back to the system \eqref{e1.1}, the first step in solving the Riemann problem in the BV class with initial data given by 
\begin{equation}
(u(x,0),\sigma(x,0))=\begin{cases} (u_{L},\sigma_{L}),\,\,\ x<0\\
                      (u_{R},\sigma_{R}),\,\,\ x>0
                     \end{cases}
\label{e1.3}
\end{equation}
would be to understand the shock-wave solutions. But this step again is not so obvious because of the following:
\begin{remark}{\textit{Volpert's product doesn't capture shock-wave type solutions for \eqref{e1.1}}}\\
For the existence of shock-wave solution to the Riemann problem using Volpert's product the following relations need to be satisfied 
$$-s[u]+[\frac{u^2}{2}]-[\sigma]=0,\\
  -s[\sigma]+\frac{1}{2}(u_{L}+u_{R})[\sigma]=0.$$ 
Here $s$ denotes the speed of the discontinuity and $[w]$ denotes the jump in the function $w$ across the discontinuity. It is easy to see 
that the above relations imply that $[\sigma]=0$ which in turn contradicts the fact that $[\sigma]\neq0$. We shall recover this result later from a different perspective. 
\end{remark}
Nevertheless we move on to classify the DLM paths which help to obtain shock-wave type solutions for the system \eqref{e1.1}. A similar 
analysis for the system \eqref{e1.2} classifies the DLM paths which allow shock-wave solutions. A remarkable correspondence between
the two as $k\rightarrow 0$ makes a strong statement so as to consider the system \eqref{e1.1} as a \textit{limiting system} to \eqref{e1.2}.
The next step of finding the rarefaction curves brings in a lot of disappointment considering the following fact.  
\begin{remark}{\textit{Rarefaction wave-type solution for the system \eqref{e1.1} is not possible.}}\\
The co-efficient matrix $A(u,\sigma)$ for the system \eqref{e1.1} is given by $A(u,\sigma)=\left( \begin{array}{cc}
u & -1  \\
0 & u  \end{array} \right).$ The existence of a rarefaction wave solution would require the existence of solution to the system
$$(A-\xi \ Id)\left( \begin{array}{c}
u' \\
\sigma' \end{array} \right)=0.$$ This in turn implies that 
$$\left( \begin{array}{c}
u' \\
\sigma' \end{array} \right)=\left( \begin{array}{c}
1 \\
0 \end{array} \right) $$(the right eigenvector) whence it follows that $\sigma$ is a constant, contradicting our assumption.

\end{remark}

\begin{remark}
 Both the above remarks are actually suggested by the analysis in \cite{j3} if we are willing to consider the system \eqref{e1.1} as a limiting
one of the system \eqref{e1.2} as $k\rightarrow 0$.
\end{remark}
Therefore the existence of solution for \eqref{e1.1} with arbitrary Riemann type initial data seems to call for a different solution space. Considering the 
parabolic degeneracy the system exhibits and the existing results for such systems, searching for singular solutions to the problem seems
to be an option. But would that be physically meaningful?\\
\\
\textbf{Heuristics suggested by physical intuition:}
Let's recall that for the system \eqref{e1.2} the unknown $u$ denotes the velocity and $\sigma$ is the stress. The constant $k$ is an elasticity constant.
Therefore putting $k=0$ would mean that the medium is \textquoteleft \textit{inelastic} \textquoteright. In an inelastic medium, any attempted motion of the constituent particles in the medium with nonconstant velocity (for each of the particles) would render a permanent deformation to the medium and hence the internal force (stress) exerted would become infinite. Mathematically this tends to suggest
that insisting on a non-zero, non-uniform (non-constant) velocity $u$ would require a blow-up or singular concentration in the stress $\sigma$.\\
\\
Motivated by the above heuristics, we prove the existence of singular solutions of the form
\begin{equation}
\begin{aligned}
u(x,t)&=u_{L}(1-H(x-\phi(t)))+u_{R} \ H(x-\phi(t)),\\
\sigma(x,t)&=\sigma_{L}(1-H(x-\phi(t)))+\sigma_{R}\ H(x-\phi(t))+e(t)\delta(x-\phi(t))
\end{aligned}
\notag
\end{equation}
for the system \eqref{e1.1} with arbitrary Riemann type initial data \eqref{e1.3}. In order to do so, following \cite{a1,d2,d3,d4,k1,k2,p2} we define the notion of generalised
delta-shock wave type solutions for the system \eqref{e1.1} and use the method of weak asymptotics to show the existence of such solutions.
We would like to remark here that the existence of weak asymptotic solutions of the system \eqref{e1.1} and \eqref{e1.2} was proved in \cite{p1}
wherein the author defined the generalised delta-shock wave solutions to be distributional limits of weak asymptotic solutions and showed 
the correspondence between the solutions as the parameter $k \rightarrow 0$. But as with the \textit{weak} solutions of partial differential
equations, a proper \textit{integral formulation} of the solution is always important. Also the utility of the results derived in \cite{p1}
in the context of the Riemann problem was not made explicit there. In this article, we deal with both these issues as well. In particular,
we suggest an integral formulation for the generalised delta-shock wave type solutions for the system \eqref{e1.1} and derive the generalised
Rankine-Hugoniot conditions from the integral identities.\\
\\
Thus from physical considerations, the existence of DLM paths giving rise to shock-wave solutions might not seem too useful. 
But from the mathematical point of view, keeping in mind the parabolic degeneracy the system \eqref{e1.1} exhibits, it surely is interesting
to find out which DLM paths are indeed applicable. Also the analysis provides a very nice correspondence between the systems \eqref{e1.1}
and \eqref{e1.2}. At the same time, the results obtained in this article seem to suggest that Volpert's product provides the \textquoteleft \textit{physical}\textquoteright   meaning to
the products involved in the systems  \eqref{e1.1} and \eqref{e1.2}.\\
\\
The structure of this article is as follows. In Section 2, we classify the applicable DLM paths and give examples of paths satifying the condition.
In Section 3, we discuss the existence of generalised delta-shock wave type solutions for the system \eqref{e1.1}. In the concluding Section 4,
we make a few remarks regarding the entropy conditions for the system \eqref{e1.1}.

\section{DLM paths and the existence of shock-wave solutions}
In this section, we give a classification of the DLM paths which allow shock-wave type solutions to the Riemann problem for the systems
\eqref{e1.1} and \eqref{e1.2}.\\
We recall that a DLM path $\phi$ is a locally Lipschitz map $\phi:[0,1]\times \mathbb{R}^2 \times \mathbb{R}^2 \rightarrow \mathbb{R}^2$
satisfying the following properties:\\
1. $\phi(0;v_{L},v_{R})=v_{L}$ and $\phi(1;v_{L},v_{R})=v_{R}$ for any $v_{L}$ and $v_{R}$ in $\mathbb{R}^2$,\\
2. $\phi(t;v,v)=v,$ for any $v$ in $\mathbb{R}^2$ and $t \in [0,1]$,\\
3. For every bounded set $\Omega$ of $\mathbb{R}^{2}$, there exists $k \geq 1$ such that 
$$\lvert \phi_{t}(t;v_{L},v_{R})-\phi_{t}(t;w_{L},w_{R})\rvert \leq
k \lvert (v_{L}-w_{L})-(v_{R}-w_{R})\rvert,$$
for every $v_{L},v_{R},w_{L},w_{R}$ in $\Omega$ and for almost every $t \in [0,1]$.\\
\\
\textbf{Examples:} It is easy to verify from the above definition that the paths $\phi=(\phi_{1},\phi_{2})$ and $\tilde \phi=(\tilde \phi_{1},\tilde \phi_{2})$ given by 
\begin{equation}
\phi(t;v_{L},v_{R}):\begin{cases} \phi_{1}(t;u_{L},u_{R})=\begin{cases} 
                                                    u_{L}+2t(u_{R}-u_{L}),\ t \in [0,\frac{1}{2}],\\
                                                    u_{R},\ t \in [\frac{1}{2},1].
                                                   \end{cases}\\
                        \phi_{2}(t;\sigma_{L},\sigma_{R})=\sigma_{L}+t(\sigma_{R}-\sigma_{L}),\ t \in [0,1]. 
                          \end{cases}
\label{e2.1}
\end{equation}
and
\begin{equation}
\tilde \phi(t;v_{L},v_{R}):\begin{cases} \tilde \phi_{1}(t;u_{L},u_{R})=\begin{cases} 
                                                    u_{L}+2t(u_{R}-u_{L}),\ t \in [0,\frac{1}{2}],\\
                                                    u_{R},\ t \in [\frac{1}{2},1]
                                                   \end{cases}\\
                        \tilde \phi_{2}(t;\sigma_{L},\sigma_{R})= \begin{cases}
                                                    \sigma_{L},\ t \in [0,\frac{1}{2}],\\
                                                    \sigma_{L}+(2t-1)(\sigma_{R}-\sigma_{L}),\ t \in [\frac{1}{2},1]\\
                                                    \end{cases}
                          \end{cases}
\label{e2.2}
\end{equation}
where $v_{L}=(u_{L},\sigma_{L}),\ v_{R}=(u_{R},\sigma_{R}) \in \mathbb{R}^2$, satisfy the above conditions and are thus DLM paths.\\
\\
Given a system $$v_{t}+A(v)v_{x}=0,\ \ v(x,t)\in \mathbb{R}^{2},\ x\in \mathbb{R},\ t>0$$ in non-conservative form with Riemann type initial data:
$$v(x,0)=v_{L} \ \text{if} \ x<0, \ v(x,0)=v_{R}\ \text{if}\ x>0,$$ and a DLM path $\phi$, it was shown in \cite{d1} that a shock wave solution exists with speed $s$
if and only if $v_{L},v_{R}$ and $s$ satisfy the following Rankine-Hugoniot (R-H) condition:
\begin{equation}
\int_{0}^{1}\{-s \ Id+A(\phi(t;v_{L},v_{R}))\}\phi_{t}(t;v_{L},v_{R})dt=0,
 \label{e2.3}
\end{equation}
where $Id$ denotes the identity matrix. The Rankine-Hugoniot condition stated above depends on the chosen DLM path $\phi$.\\
We remark that the matrix $A$ for the systems \eqref{e1.1} and \eqref{e1.2} are given by 
$A(u,\sigma)=\left( \begin{array}{cc}
u & -1  \\
0 & u  \end{array} \right)$ 
and 
$A(u,\sigma)=\left( \begin{array}{cc}
u & -1  \\
-k^2 & u  \end{array} \right)$
respectively.\\
\\
We now derive a condition on a DLM path $\phi=(\phi_{1},\phi_{2})$ in order that it might satisfy the R-H condition \eqref{e2.3} for the
systems \eqref{e1.1} and \eqref{e1.2} for given Riemann type initial data 
\begin{equation}
v(x,0)=(u(x,0),\sigma(x,0))=\begin{cases}
                             v_{L}=(u_{L},\sigma_{L}),\ x<0,\\
                             v_{R}=(u_{R},\sigma_{R}),\ x>0.
                            \end{cases}
\label{e2.4} 
\end{equation}
Lax's admissibility condition applied to the systems \eqref{e1.1} and \eqref{e1.2} implies that for a shock-wave type solution we must have
$u_{L}>u_{R}$.
\begin{theorem}
Given initial states $v_{L}=(u_{L},\sigma_{L})$ and $v_{R}=(u_{R},\sigma_{R})$ and a DLM path $\phi=(\phi_{1},\phi_{2})$, the R-H condition \eqref{e2.3}
for the system \eqref{e1.1} is satisfied if 
\begin{equation}
\int_{0}^{1} \phi_{1}(t;u_{L},u_{R}) (\phi_{2})_{t}(t;\sigma_{L},\sigma_{R}) dt=[\sigma].\frac{[\frac{u^2}{2}]-[\sigma]}{[u]},
 \label{e2.5}
\end{equation}
where $[u]=(u_{R}-u_{L})$ denotes the jump in $u$ and $s=\frac{[\frac{u^2}{2}]-[\sigma]}{[u]}$ is the speed of the shock. 
\end{theorem}
\begin{proof}
We recall that for the system \eqref{e1.1} we have
$A(u,\sigma)=\left( \begin{array}{cc}
u & -1  \\
0 & u  \end{array} \right)$. Substituting this in \eqref{e2.3}, we obtain 
$\int_{0}^{1} \left( \begin{array}{cc}
-s+\phi_{1} & -1  \\
0 & -s+\phi_{1}  \end{array} \right)
\left( \begin{array}{c}
(\phi_{1})_{t}  \\
(\phi_{2})_{t}  \end{array} \right)= 0$.\\
Thus we have the relations
\begin{equation}
 \int_{0}^{1} -s(\phi_{1})_{t}+\phi_{1}(\phi_{1})_{t}-(\phi_{2})_{t} \ dt= 0,
\label{e2.6}
\end{equation}
and
\begin{equation}
 \int_{0}^{1} -s(\phi_{2})_{t}+\phi_{1}(\phi_{2})_{t} \ dt= 0.
 \label{e2.7}
\end{equation}
Now \eqref{e2.6} on simplification (using the fact $\phi_{1}(0)=u_{L},\ \phi_{1}(1)=u_{R},\ \phi_{2}(0)=\sigma_{L},\ \phi_{2}(1)=\sigma_{R}$) gives $s=\frac{[\frac{u^2}{2}]-[\sigma]}{[u]}$. Substituting this in 
\eqref{e2.7} we obtain 
$$\int_{0}^{1} \phi_{1}(t;u_{L},u_{R}) (\phi_{2})_{t}(t;\sigma_{L},\sigma_{R}) dt=[\sigma].\frac{[\frac{u^2}{2}]-[\sigma]}{[u]}.$$
\end{proof}

\begin{remark}
It can be easily seen that the straight line path giving rise to Volpert's product doesn't satisfy the relation \eqref{e2.5}. Let
$\phi_{1}(t)=u_{L}+t(u_{R}-u_{L})$ and $\phi_{2}(t)=\sigma_{L}+t(\sigma_{R}-\sigma_{L})$ be the two components of the straight line path $\phi$.
Substituting these in \eqref{e2.5}, we get  
\begin{equation}
\begin{aligned}
\int_{0}^{1}(u_{L}+t(u_{R}-u_{L}))dt.[\sigma] &=\frac{[\frac{u^2}{2}]-[\sigma]}{[u]}.[\sigma]\\
\Longrightarrow u_{L}+\frac{[u]}{2} &=\frac{[\frac{u^2}{2}]-[\sigma]}{[u]}\\
\Longrightarrow \frac{[\sigma]}{[u]}&=0\\
\Longrightarrow [\sigma]&=0,
\end{aligned}
 \notag
\end{equation}
which is a contradiction.
\end{remark}

Proceeding similarly as in the above Theorem $2.1$, we can prove the corresponding result for the system \eqref{e1.2}.
\begin{theorem}
 Given initial states $v_{L}=(u_{L},\sigma_{L})$ and $v_{R}=(u_{R},\sigma_{R})$ and a DLM path $\phi=(\phi_{1},\phi_{2})$, the R-H condition \eqref{e2.3}
for the system \eqref{e1.2} is satisfied if 
\begin{equation}
\int_{0}^{1} \phi_{1}(t;u_{L},u_{R}) (\phi_{2})_{t}(t;\sigma_{L},\sigma_{R}) dt=[\sigma].\frac{[\frac{u^2}{2}]-[\sigma]}{[u]}+k^2[u],
 \label{e2.8}
\end{equation}
where $[u]=(u_{R}-u_{L})$ denotes the jump in $u$ and $s=\frac{[\frac{u^2}{2}]-[\sigma]}{[u]}$ is the speed of the shock. 
\end{theorem}
\begin{proof}
 A similar calculation as in the proof of Theorem $2.1$ with the corresponding matrix 
$A(u,\sigma)=\left( \begin{array}{cc}
u & -1  \\
-k^2 & u  \end{array} \right)$
for the system \eqref{e1.2} gives the result.
\end{proof}

\begin{remark}
 The relations \eqref{e2.5} and \eqref{e2.8} exhibit a nice correspondence between the systems \eqref{e1.1} and \eqref{e1.2}. In particular,
 putting $k=0$ in \eqref{e2.8} we recover the relation \eqref{e2.5} for the system \eqref{e1.1}.
\end{remark}

Next we prove the existence of DLM paths satisfying the conditions \eqref{e2.5} and \eqref{e2.8} for the systems \eqref{e1.1}
and \eqref{e1.2} respectively. In particular, we show that given a left state $v_{L}=(u_{L},\sigma_{L})$ the paths $\phi$ and $\tilde \phi$
given by \eqref{e2.1} and \eqref{e2.2} give rise to shock curves passing through $v_{L}$. Thus for any state $v_{R}=(u_{R},\sigma_{R})$
(with $u_{R}<u_{L}$)  lying on these curves, we can solve the Riemann problem using a shock wave with speed $s$ mentioned above.
\begin{theorem}
 Given a left state $v_{L}=(u_{L},\sigma_{L})$, the DLM path $\phi$ defined in \eqref{e2.1} gives a shock-wave solution for the system \eqref{e1.1}
with the right states $v_{R}=(u_{R},\sigma_{R})$ on the shock curve
\begin{equation}
S_{1}:\sigma=\sigma_{L}-\frac{1}{4}(u-u_{L})^2,\ u<u_{L}.
\label{e2.9}
\end{equation}
\end{theorem}
\begin{proof}
The result follows from a straightforward calculation by substituting $\phi_{1}$ and $\phi_{2}$ from \eqref{e2.1} in \eqref{e2.5}.  
\end{proof}
Similar to the above theorem for the system \eqref{e1.2} we have the following result.
\begin{theorem}
 Given a left state $v_{L}=(u_{L},\sigma_{L})$, the DLM path $\phi$ defined in \eqref{e2.1} gives a shock-wave solution for the system \eqref{e1.2}
with the right states $v_{R}=(u_{R},\sigma_{R})$ on the shock curves
\begin{equation}
S_{1}:\sigma=\sigma_{L}-\frac{1}{8}((u-u_{L})^2-\sqrt{(u-u_{L})^4+64k^2(u-u_{L})^2}),\ u<u_{L},
\label{e2.10}
\end{equation}
\begin{equation}
S_{2}:\sigma=\sigma_{L}-\frac{1}{8}((u-u_{L})^2+\sqrt{(u-u_{L})^4+64k^2(u-u_{L})^2}),\ u<u_{L},
\label{e2.11}
\end{equation}
\end{theorem}
\begin{proof}
 Substituting $\phi_{1}$ and $\phi_{2}$ from \eqref{e2.1} in \eqref{e2.8} we obtain the following quadratic equation in $[\sigma]$:
$$4[\sigma]^2+(u_{R}-u_{L})^2[\sigma]-4k^2(u_{R}-u_{L})^2=0.$$
The above equation when solved gives us the required expressions for $S_{1}$ and $S_{2}$.
\end{proof}
\begin{remark}
 As $k$ tends to $0$, the shock curve $S_{2}$ defined in \eqref{e2.11} for the system \eqref{e1.2} tends to the shock curve for the system 
\eqref{e1.1} given by \eqref{e2.9}, while the other shock curve $S_{1}$ degenerates.
\end{remark}

\begin{remark}
 The path $\phi$ is an example of a DLM path satisfying \eqref{e2.5} with $\phi_{2}$ being a straight line. In particular, it follows from
the geometrical interpretation of \eqref{e2.5} that if $\phi$ is a DLM path such that $\phi_{2}$ is a straight line but area under the curve $\phi_{1}$
is different from the area under the straight line connecting $u_{L}$ and $u_{R}$, then such a DLM path would give rise to a shock-wave solution
to the Riemann problem for the system \eqref{e1.1}.  
\end{remark}

Next we state the analogous results for the DLM path $\tilde \phi$ defined in \eqref{e2.2}.
\begin{theorem}
 Given a left state $v_{L}=(u_{L},\sigma_{L})$, the DLM path $\tilde \phi$ defined in \eqref{e2.2} gives a shock-wave solution for the system \eqref{e1.1}
with the right states $v_{R}=(u_{R},\sigma_{R})$ on the shock curve
\begin{equation}
S_{1}:\sigma=\sigma_{L}-\frac{1}{2}(u-u_{L})^2,\ u<u_{L}.
\label{e2.12}
\end{equation}
\end{theorem}
\begin{proof}
 The result follows from a straightforward calculation by substituting $\tilde \phi_{1}$ and $\tilde \phi_{2}$ from \eqref{e2.2} in \eqref{e2.5}. 
\end{proof}

\begin{theorem}
 Given a left state $v_{L}=(u_{L},\sigma_{L})$, the DLM path $\tilde \phi$ defined in \eqref{e2.2} gives a shock-wave solution for the system \eqref{e1.2}
with the right states $v_{R}=(u_{R},\sigma_{R})$ on the shock curves
\begin{equation}
S_{1}:\sigma=\sigma_{L}-\frac{1}{4}((u-u_{L})^2-\sqrt{(u-u_{L})^4+16k^2(u-u_{L})^2}),\ u<u_{L},
\label{e2.13}
\end{equation}
\begin{equation}
S_{2}:\sigma=\sigma_{L}-\frac{1}{4}((u-u_{L})^2+\sqrt{(u-u_{L})^4+16k^2(u-u_{L})^2}),\ u<u_{L}.
\label{e2.14}
\end{equation}
\end{theorem}
\begin{proof}
Substituting $\tilde \phi_{1}$ and $\tilde \phi_{2}$ from \eqref{e2.2} in \eqref{e2.8} we obtain the following quadratic equation in $[\sigma]$:
$$2[\sigma]^2+(u_{R}-u_{L})^2[\sigma]-2k^2(u_{R}-u_{L})^2=0.$$
The above equation when solved gives us the required expressions for $S_{1}$ and $S_{2}$.
\end{proof}

\begin{remark}
 As $k$ tends to $0$, the shock curve $S_{2}$ defined in \eqref{e2.14} for the system \eqref{e1.2} tends to the shock curve for the system 
\eqref{e1.1} given by \eqref{e2.12}, while the other shock curve $S_{1}$ degenerates.
\end{remark}

\section{Generalised Delta-shock wave type solutions and the Riemann problem}
In this section, we prove the existence of generalised delta-shock wave type solution for the system \eqref{e1.1} with Riemann type initial
data and discuss its role in solving the Riemann problem. Here we use the method of weak asymptotics (see \cite{a1,d2,d3,d4,k1,k2,p2}) to construct the solution.
As already mentioned in the introduction, the existence of weak asymptotics solution for the system \eqref{e1.1} was proved in \cite{p1} but an integral formulation
for the generalised delta-shock wave type solution wasn't given nor was its role in resolving the Riemann problem discussed explicitly. 
For the sake of completeness, we include here part of the results and calculations from \cite{p1}.

\subsection{The method of weak asymptotics}
Let us denote by $\mathcal{D}$ and $\mathcal{D}^{\prime}$ the space of smooth functions with compact support and the space of distributions respectively.
By $O_{\mathcal{D}^{\prime}}(\epsilon^{\alpha})$ we denote the collection of distributions $f(x,t,\epsilon) \in \mathcal{D}^{\prime}(\mathbb{R})$ such that for any test function
$\varphi(x) \in \mathcal{D}(\mathbb{R})$ the estimate $$\langle f(x,t,\epsilon),\varphi(x)\rangle=O(\epsilon^{\alpha})$$
holds and is uniform with respect to $t$. We interpret the relation $o_{\mathcal{D}^{\prime}}(\epsilon^{\alpha})$ in a similar manner.\\
\textbf{Definition $3.1$ : }(see \cite{k1,k2,p1})
A pair of smooth complex-valued (real-valued) functions $(u(x,t,\epsilon),\sigma(x,t,\epsilon))$ is called a \textit{weak asymptotic solution}
of the system \eqref{e1.1} with the initial data $(u(x,0),\sigma(x,0))$ if
\begin{equation}
 \begin{aligned}
  &u_{t}(x,t,\epsilon)+u(x,t,\epsilon)u_{x}(x,t,\epsilon)-\sigma_{x}(x,t,\epsilon)=o_{\mathcal{D}^{\prime}}(1), \\
  &\sigma_{t}(x,t,\epsilon)+u(x,t,\epsilon)\sigma_{x}(x,t,\epsilon)=o_{\mathcal{D}^{\prime}}(1),\\
  &u(x,0,\epsilon)-u(x,0)=o_{\mathcal{D}^{\prime}}(1),\\
  &\sigma(x,0,\epsilon)-\sigma(x,0)=o_{\mathcal{D}^{\prime}}(1),\ \epsilon \rightarrow 0.
   \end{aligned}
 \label{e3.1}
\end{equation}
%
%
Next we introduce the notion of \textit{generalised delta-shock wave type solution} in the context of the system \eqref{e1.1}.
In the context of system of conservation laws, this notion (in its integral form) was introduced in \cite{d4}.\\ 
\textbf{Definition $3.2$ :}
Let $u \in BV(\mathbb{R}\times(0,\infty);\mathbb{R})$ and $$\sigma(x,t)=\tilde \sigma(x,t)+e(x,t)\delta(x-\phi(t)),$$ where 
$\tilde \sigma \in BV(\mathbb{R} \times (0,\infty);\mathbb{R})$ and $e$ is a smooth function. The pair $(u,\sigma)$ is called a \textit{generalised delta-shock wave type}
solution of \eqref{e1.1} with initial conditions $u(0)$ and $\sigma(0)=\sigma_{0}+\sigma_{1}H(-x)$, if the following integral identities hold for
all $\theta(x,t)\in \mathcal{D}(\mathbb{R} \times [0,\infty))$:
\begin{equation}
\begin{aligned}
\int_{0}^{\infty} \int_{\mathbb{R}} (u \theta_{t}+(\frac{u^2}{2}-\tilde \sigma)\theta_{x})\ dx\ dt+\int_{\mathbb{R}}u(0)\theta(x,0)\ dx&=0,\\
\int_{0}^{\infty} \int_{\mathbb{R}} (\tilde \sigma_{t}+\hat u \tilde \sigma_{x})\theta \ dx \ dt-\int_{\Gamma}e(x,t)\frac{\partial \theta(x,t)}{\partial l}\ dl &=0.
\end{aligned}
 \label{e3.2}
\end{equation}
Here $\sigma_{0},\sigma_{1}$ are constants, $\hat u(x)=\int_{0}^{1}(u(x-)+t(u(x+)-u(x-)))dt$ and $\frac{\partial \theta(x,t)}{\partial l}$ is the tangential derivative on the graph
$\Gamma=\{(x,t):x=\phi(t)\}$.
 

\begin{remark}
 It would be important to note here that in the second identity in \eqref{e3.2}, the first integral is the weak formulation
using the Volpert's product. We have already seen that the Volpert's product isn't sufficient to get solution for the system. Therefore 
we augment it by allowing a delta-term in the form $\int_{\Gamma}e(x,t)\frac{\partial \theta(x,t)}{\partial l}\ dl $.
\end{remark}

\begin{remark}
 We can also rewrite the second identity in \eqref{e3.2} in the form $$\int_{0}^{\infty} \int_{\mathbb{R}} \tilde \sigma \theta_{t}\ dx \ dt
- \int_{0}^{\infty} \int_{\mathbb{R}} \hat u \tilde \sigma_{x} \theta \ dx \ dt + \int_{\Gamma} e(x,t) \frac{\partial \theta(x,t)}{\partial l}\ dl 
+ \int_{\mathbb{R}} \sigma(0) \theta(x,0)\ dx=0.$$
\end{remark}

Since we are interested in the solution of the Riemann problem, let us consider initial data of the form 
\begin{equation}
 \begin{aligned}
  &u(x,0)=u_{0}+u_{1}H(-x),\\
  &\sigma(x,0)=\sigma_{0}+\sigma_{1}H(-x),
 \end{aligned}
 \label{e3.3}
 \end{equation}
where $u_{0},u_{1},\sigma_{0},\sigma_{1}$ are constants. We propose a singular ansatz to \eqref{e1.1},\eqref{e3.3} of the form
\begin{equation}
 \begin{aligned}
  &u(x,t)=u_{0}+u_{1}H(-x+\phi(t)),\\
  &\sigma(x,t)=\sigma_{0}+\sigma_{1}H(-x+\phi(t))+e(t)\delta(x-\phi(t)),
 \end{aligned}
\label{e3.4}
\end{equation}
where $e(t)$ is a smooth function to be determined. In order to apply the method of weak asymptotics, we start with appropriate regularizations
$H_{u},H_{\sigma}$ of the Heaviside function and $\delta(.,\epsilon)$ of the delta distribution and choose proper \textit{correction} terms $R_{u},R_{\sigma}$ so as to propose a smooth ansatz for the weak asymptotic 
solution in the following form
\begin{equation}
 \begin{aligned}
 & u(x,t,\epsilon)=u_{0}+u_{1}H_{u}(-x+\phi(t),\epsilon)+R_{u}(x,t,\epsilon),\\
 & \sigma(x,t,\epsilon)=\sigma_{0}+\sigma_{1}H_{\sigma}(-x+\phi(t),\epsilon)+e(t)\delta(x-\phi(t),\epsilon)+R_{\sigma}(x,t,\epsilon).
 \end{aligned}
 \label{e3.5}
\end{equation}
The correction terms $R_{u},R_{\sigma}$ satisfy the conditions
 $$R_{i}(x,t,\epsilon)=o_{\mathcal{D}^{\prime}}(1),\  (R_{i})_{t}(x,t,\epsilon)=o_{\mathcal{D}^{\prime}}(1),\ 
\epsilon \rightarrow 0,\ i=u,\sigma.$$
We then substitute the smooth ansatz in the left-hand side of the system \eqref{e1.1} and determine $\phi(t),e(t)$ so that the smooth ansatz 
forms a weak asymptotic solution of \eqref{e1.1},\eqref{e3.3}. Then by passing to the limit, we show that the singular ansatz \eqref{e3.4}
indeed satisfies \eqref{e3.2}.

\subsection{Generalised Rankine-Hugoniot conditions from the integral identities}
We now derive the \textit{generalised Rankine-Hugoniot conditions} satisfied by a generalised delta-shock wave type solution for the system
\eqref{e1.1}. Henceforth we use the convention that across a discontinuity, the jump $[v]$ of a function is given by $[v]=v_{L}-v_{R}$,
 where $v_{L},v_{R}$ denote respectively the left and right values across the discontinuity.

\begin{theorem}
 Let $\Omega \subset \mathbb{R} \times (0,\infty)$ be a domain in $\mathbb{R}^2$ and let $\Gamma=\{(x,t):x=\phi(t) \}$ be a smooth curve that divides 
$\Omega$ into the two halves $\Omega_{-} = \{(x,t): x-\phi(t)<0 \}$ and $\Omega_{+} =\{(x,t): x-\phi(t)>0 \} $. Let $(u,\sigma)$ be a
generalised delta-shock wave type solution \eqref{e3.2} of \eqref{e1.1} where $\sigma$ is of the form 
$$\sigma(x,t)=\sigma_{0}+\sigma_{1}H(-x+\phi(t))+e(x,t)\delta(x-\phi(t)),$$ where $\sigma_{0},\sigma_{1}$ are constants.
Then $(u,\sigma)$ satisfies the following 
generalised Rankine-Hugoniot conditions for delta-shocks along the discontinuity curve $\Gamma:$
\begin{equation}
\dot \phi(t)=\frac{[\frac{u^2}{2}]-[\sigma]}{[u]},\ \ 
\dot e(t) =\frac{[\sigma]^2}{[u]}.
 \label{e3.6}
\end{equation}
\end{theorem}

\begin{proof}
The condition on $\phi(t)$ follows by a standard argument using integration by parts in the first identity in \eqref{e3.2}. \\
Next we would like to note that on the curve $\Gamma$, $e$ can be considered as a function of the single variable $t$:  $e(t)=e(\phi(t),t)$ and
that $$\int_{\Gamma} e(x,t) \frac{\partial \theta(x,t)}{\partial l} dl=\int_{0}^{\infty} e(t) \frac{d \theta(\phi(t),t)}{dt} dt,$$
where $\frac{d \theta(x,t)}{dt}=\theta_{t}(x,t)+\dot \phi(t) \theta_{x}(x,t)$.

Now $\tilde \sigma_{x}=(\sigma_{0}+\sigma_{1}H(-x+\phi(t)))_{x}=-\sigma_{1} \delta(-x+\phi(t))$
and $$\tilde \sigma_{t}=(\sigma_{0}+\sigma_{1}H(-x+\phi(t)))_{t}= \sigma_{1} \dot \phi(t) \delta(-x+\phi(t)).$$
Therefore using the fact that $\hat u(x)=(u_{0}+\frac{u_{1}}{2})$ on $\Gamma$, we get 
$$\int_{0}^{\infty} \int_{\mathbb{R}}(\tilde \sigma_{t}+\hat u \tilde \sigma_{x})\theta \ dx \ dt= \int_{0}^{\infty} (\sigma_{1}
\dot \phi(t)-\sigma_{1}(u_{0}+\frac{u_{1}}{2}))\theta(\phi(t),t)) \ dt.$$
The second identity in \eqref{e3.2} thus becomes 
\begin{equation}
\begin{aligned}
0&=\int_{0}^{\infty} \int_{\mathbb{R}}(\tilde \sigma_{t}+\hat u \tilde \sigma_{x})\theta \ dx \ dt - \int_{0}^{\infty} e(t) \frac{d \theta(\phi(t),t)}{dt}  \ dt \\
&= \int_{0}^{\infty} (\sigma_{1} \dot \phi(t)-\sigma_{1}(u_{0}+\frac{u_{1}}{2})+\frac{de(\phi(t),t)}{dt})\theta(\phi(t),t)) \ dt.
\end{aligned}
\notag
\end{equation}
Since the above identity is satisfied for all $\theta \in \mathcal{D}(\mathbb{R} \times [0,\infty)) $, we have 
$$\sigma_{1} \dot \phi(t)-\sigma_{1}(u_{0}+\frac{u_{1}}{2})+\frac{de(\phi(t),t)}{dt}=0 $$
whereby it follows (using the condition on $\phi(t)$ already obtained) that
$$\dot e(t):=\frac{d e(\phi(t),t)}{dt}=\sigma_{1}(u_{0}+\frac{u_{1}}{2})-\sigma_{1}(u_{0}+\frac{u_{1}}{2}-\frac{\sigma_{1}}{u_{1}})=\frac{\sigma_{1}^2}{u_{1}}=\frac{[\sigma]^2}{[u]} .$$ 
Thus we obtain the generalised Rankine-Hugoniot conditions in the form \eqref{e3.6}.
\end{proof}

\subsection{Existence of generalised delta-shock wave type solution}
Next we quickly take a look at the regularizations and correction terms to be chosen in our case (see \cite{p1}).\\
Let $\omega:\mathbb{R}\rightarrow \mathbb{R}$ be a non-negative, smooth, even function with support in $(-1,1)$ and satisfying
$$\int_{\mathbb{R}}\omega(x)dx=1.$$
Let $\omega_{0}=\int_{\mathbb{R}}\omega^{2}(x)dx$ and let 
\begin{equation}
 R(x,t,\epsilon)=\frac{1}{\sqrt{\epsilon}}\omega(\frac{x-2\epsilon}{\epsilon}),\ \
\delta(x,\epsilon)=\frac{1}{\epsilon}\omega(\frac{x+2\epsilon}{\epsilon}).
\label{e3.7}
\end{equation}
We take $R_{u}=p(t)R(.,\epsilon)$ and $R_{\sigma}=0$, where $p(t)$ is a smooth function to be determined.
Further we take $H_{u}(x,\epsilon)=H_{\sigma}(x,\epsilon)=H(x,\epsilon)$ where $H(x,\epsilon)$ is defined as
\begin{equation}
H(x,\epsilon)=\begin{cases} 0,\ \ x\leq -4\epsilon\\
                           c,\ \ -3\epsilon\leq x \leq 3\epsilon\\
                           1,\ \ x\geq 4\epsilon
                           \end{cases}
\label{e3.8}
\end{equation}
and is continued smoothly in the regions $(-4\epsilon,-3\epsilon)$ and $(3\epsilon,4\epsilon)$. Here we take 
the constant $c=(\frac{1}{2}-\frac{\sigma_{1}}{u_{1}^2}).$
Thus our smooth ansatz takes the form
\begin{equation}
 \begin{aligned}
 & u(x,t,\epsilon)=u_{0}+u_{1}H(-x+\phi(t),\epsilon)+p(t)R(x-\phi(t),\epsilon),\\
 & \sigma(x,t,\epsilon)=\sigma_{0}+\sigma_{1}H(-x+\phi(t),\epsilon)+e(t)\delta(x-\phi(t),\epsilon).
 \end{aligned}
 \label{e3.9}
\end{equation}
We then have the asymptotics given by \\
\textbf{Lemma :}(see \cite{p1})
Choosing the regularizations and corrections as described above, we have the following weak
asymptotic expansions:
\begin{equation}
 \begin{aligned}
 &R(x,\epsilon)=o_{\mathcal{D}^{\prime}}(1),\ \ R_{x}(x,\epsilon)=o_{\mathcal{D}^{\prime}}(1),\\
 &R^{2}(x,\epsilon)=\omega_{0}\delta(x)+o_{\mathcal{D}^{\prime}}(1),\\
 &R(x,\epsilon) R_{x}(x,\epsilon)=\frac{1}{2}\omega_{0}\delta^{\prime}(x)+o_{\mathcal{D}^{\prime}}(1),\\
 &\delta(x,\epsilon)=\delta(x)+o_{\mathcal{D}^{\prime}}(1),\ \ \delta_{x}(x,\epsilon)=\delta^{\prime}(x)+o_{\mathcal{D}^{\prime}}(1),\\
 &R(x,\epsilon)\delta(x,\epsilon)=0,\ R(x,\epsilon) \delta_{x}(x,\epsilon)=0,\\
 &H(x,\epsilon)=H(x)+o_{\mathcal{D}^{\prime}}(1),\  H_{x}(x,\epsilon)=\delta(x)+o_{\mathcal{D}^{\prime}}(1),\\ 
 &H(x,\epsilon) H_{x}(x,\epsilon)=\frac{1}{2}\delta(x)+o_{\mathcal{D}^{\prime}}(1),\\
 &H(x,\epsilon) R_{x}(-x,\epsilon)=o_{\mathcal{D}^{\prime}}(1),\ R(-x,\epsilon) H_{x}(x,\epsilon)=o_{\mathcal{D}^{\prime}}(1),\\
 &H(x,\epsilon) \delta_{x}(-x,\epsilon)=c\delta^{\prime}(-x)+o_{\mathcal{D}^{\prime}}(1),\ \epsilon \rightarrow 0.
 \end{aligned}
\notag
 \end{equation}

The next theorem gives us the existence of a weak asymptotic solution of the system \eqref{e1.1} in the form \eqref{e3.9}.
\begin{theorem}(see \cite{p1})
 For $t\in [0,\infty)$, the Cauchy problem \eqref{e1.1},\eqref{e3.3} has a weak asymptotic solution \eqref{e3.9} with $\phi(t),e(t)\ and\ p(t)$
given by the relations 
\begin{equation}
\begin{aligned}
&\dot{\phi}(t)=\frac{[\frac{u^{2}}{2}]-[\sigma]}{[u]},\  \dot{e}(t)=\frac{[\sigma]^{2}}{[u]},\\
&\frac{1}{2}p^{2}(t)\omega_{0}-e(t)=0.
\end{aligned}
\label{e3.10}
\end{equation}
\end{theorem}

\begin{remark}
 It would be interesting to note that the expressions for $\dot{\phi}(t)$ and $\dot{e}(t)$ are the generalised Rankine-Hugoniot conditions derived before.
\end{remark}

Finally, we then have the following desired result
\begin{theorem}
 For $t\in [0,\infty)$, the Cauchy problem \eqref{e1.1},\eqref{e3.3} has a generalised delta-shock wave type solution of the form \eqref{e3.4} with $\phi(t)\ and \ e(t)$
given by the relations 
\begin{equation}
\phi(t)=\frac{[\frac{u^{2}}{2}]-[\sigma]}{[u]}t,\  e(t)=\frac{[\sigma]^2}{[u]}t.
\label{e3.11}
\end{equation}
\end{theorem}

\begin{proof}
 It is sufficient to show that as $\epsilon \rightarrow 0$, the weak asymptotic solution \eqref{e3.9} with $\phi(t),e(t)$ given by the relations \eqref{e3.11}
satisfy the integral identities \eqref{e3.2}. Since $(u(x,t,\epsilon),\sigma(x,t,\epsilon))$ is a weak asymptotic solution of the system 
\eqref{e1.1} with initial data \eqref{e3.3}, it follows that for every $\theta \in \mathcal{D}(\mathbb{R} \times [0,\infty))$,
\begin{equation}
 \begin{aligned}
&\int_{0}^{\infty} \int_{\mathbb{R}} (u_{t}(x,t,\epsilon)+u(x,t,\epsilon)u_{x}(x,t,\epsilon)-\sigma_{x}(x,t,\epsilon)) \theta(x,t)\ dx \ dt \rightarrow 0,\\
&\int_{0}^{\infty} \int_{\mathbb{R}} (\sigma_{t}(x,t,\epsilon)+u(x,t,\epsilon)\sigma_{x}(x,t,\epsilon)) \theta(x,t)\ dx \ dt \rightarrow 0,\\
&\int_{\mathbb{R}} (u(x,0,\epsilon) -u(x,0))\theta(x,0) dx \rightarrow 0,\\
&\int_{\mathbb{R}} (\sigma(x,0,\epsilon)-\sigma(x,0))\theta(x,0) dx \rightarrow 0,\  \text{as}\  \epsilon \rightarrow 0.
 \end{aligned}
\label{e3.12}
\end{equation}
Now,
\begin{equation}
 \begin{aligned}
  &\int_{0}^{\infty} \int_{\mathbb{R}} (u_{t}(x,t,\epsilon)+u(x,t,\epsilon)u_{x}(x,t,\epsilon)-\sigma_{x}(x,t,\epsilon))\theta(x,t)\ dx \ dt \\
  &=-\int_{0}^{\infty} \int_{\mathbb{R}} u(x,t,\epsilon) \theta_{t}(x,t)\ dx \ dt\\ 
&-\int_{0}^{\infty} \int_{\mathbb{R}} (\frac{u(x,t,\epsilon)^2}{2}-\sigma(x,t,\epsilon))\theta_{x}(x,t)\ dx \ dt-\int_{\mathbb{R}}u(x,0,\epsilon) \theta(x,0)\ dx.
 \end{aligned}
\label{e3.13}
\end{equation}
But 
\begin{equation}
 \begin{aligned}
 \frac{u^2(x,t,\epsilon)}{2}-\sigma(x,t,\epsilon)&=\frac{1}{2}(u_{0}^2+u_{1}^2 H^2(-x+\phi(t),\epsilon)+p^2(t) R^2(x-\phi(t),\epsilon) \\ 
&\ \ \ \ \ +2 u_{0} u_{1} H(-x+\phi(t),\epsilon)+2 u_{0} p(t) R(x-\phi(t),\epsilon)\\ 
&\ \ \ \ \ +2 u_{1} p(t) H(-x+\phi(t),\epsilon) R(x-\phi(t),\epsilon)) -\sigma_{0} \\ 
&\ \ \ \ \ -\sigma_{1}H(-x+\phi(t),\epsilon)-e(t) \delta(x-\phi(t),\epsilon)\\
&=(\frac{u_{0}^2}{2}-\sigma_{0})+(\frac{u_{1}^2}{2}+u_{0} u_{1} - \sigma_{1})H(-x+\phi(t),\epsilon) \\ 
&\ \ \ \ \ +\frac{1}{2}p^2(t)R^2(x-\phi(t),\epsilon) -e(t)\delta(x-\phi(t),\epsilon) \\ 
&\ \ \ \ \ +u_{0}p(t)R(x-\phi(t),\epsilon)\\ 
&\ \ \ \ \ +u_{1}p(t)H(-x+\phi(t),\epsilon)R(x-\phi(t),\epsilon)\\
&=(\frac{u_{0}^2}{2}-\sigma_{0})+(\frac{u_{1}^2}{2}+u_{0} u_{1} - \sigma_{1})H(-x+\phi(t)) \\ 
&\ \ \ \  +(\frac{1}{2}\omega_{0} p^2(t)-e(t))\delta(x-\phi(t))+ o_{\mathcal{D}^{\prime}}(1) \\
&=(\frac{u_{0}^2}{2}-\sigma_{0})+(\frac{u_{1}^2}{2}+u_{0} u_{1} - \sigma_{1})H(-x+\phi(t))+o_{\mathcal{D}^{\prime}}(1) \\
&=(\frac{u^2(x,t)}{2}-\tilde \sigma(x,t))+ o_{\mathcal{D}^{\prime}}(1).
 \end{aligned}
\notag
\end{equation}

Therefore passing to the limit as $\epsilon \rightarrow 0$ in \eqref{e3.13} and using \eqref{e3.10} and \eqref{e3.12}, we get
\begin{equation}
0=\int_{0}^{\infty} \int_{\mathbb{R}} (u \theta_{t}+ (\frac{u^2}{2}-\tilde \sigma) \theta_{x})\ dx \ dt + \int_{\mathbb{R}} u(x,0) \theta(x,0)\ dx.
 \label{e3.14}
\end{equation}
 
Next, we observe that
\begin{equation}
 \begin{aligned}
  &\int_{0}^{\infty} \int_{\mathbb{R}} u(x,t,\epsilon) \sigma_{x}(x,t,\epsilon) \theta(x,t)\ dx \ dt \\ 
&\ \ \ \ \ \ \ \ =\int_{0}^{\infty} \int_{\mathbb{R}} u(x,t,\epsilon) (\tilde \sigma(x,t,\epsilon)+e(x,t) \delta(x-\phi(t),\epsilon))_{x} \theta(x,t) \ dx \ dt \\
&\ \ \ \ \ \ \ \ =\int_{0}^{\infty} \int_{\mathbb{R}} u(x,t,\epsilon) \tilde \sigma_{x}(x,t,\epsilon) \theta(x,t)\ dx \ dt \\ 
&\ \ \ \ \ \ \ \ \ \ \ \ \ +\int_{0}^{\infty} \int_{\mathbb{R}} u(x,t,\epsilon) (e(x,t) \delta(x-\phi(t),\epsilon))_{x} \theta(x,t) \ dx \ dt \\
&\ \ \ \ \ \ \ \ = A\ + \ B,
  \end{aligned}
\notag
\end{equation}
where $A=\int_{0}^{\infty} \int_{\mathbb{R}} u(x,t,\epsilon) \tilde \sigma_{x}(x,t,\epsilon) \theta(x,t)\ dx \ dt $, and 
$$B= \int_{0}^{\infty} \int_{\mathbb{R}} u(x,t,\epsilon) (e(x,t) \delta(x-\phi(t),\epsilon))_{x} \theta(x,t) \ dx \ dt. $$

Also, 
\begin{equation}
 \begin{aligned}
 \int_{0}^{\infty} \int_{\mathbb{R}} \sigma_{t}(x,t,\epsilon) \theta(x,t)\ dx \ dt &= \int_{0}^{\infty} \int_{\mathbb{R}} \tilde \sigma_{t}(x,t,\epsilon)
\theta(x,t)\ dx \ dt \\  
&\ \ \ \ \ \ \ \ \ \ \ \ \  + \int_{0}^{\infty} \int_{\mathbb{R}} (e(x,t) \delta(x-\phi(t),\epsilon))_{t} \theta(x,t)\ dx \ dt \\
 &= \int_{0}^{\infty} \int_{\mathbb{R}} \tilde \sigma_{t}(x,t) \theta(x,t)\ dx \ dt+\  C\  + o(1),  
 \end{aligned}
\notag
\end{equation}
where $C=\int_{0}^{\infty} \int_{\mathbb{R}} (e(x,t) \delta(x-\phi(t),\epsilon))_{t} \theta(x,t)\ dx \ dt.$\\
Therefore,
\begin{equation}
 \begin{aligned}
  \int_{0}^{\infty} \int_{\mathbb{R}}(\sigma_{t}(x,t,\epsilon)+u(x,t,\epsilon)\sigma_{x}(x,t,\epsilon))\theta(x,t)\ dx \ dt 
   &= \int_{0}^{\infty} \int_{\mathbb{R}} \tilde \sigma_{t}(x,t) \theta(x,t) \ dx \ dt \\ 
  &\ \ \ \ \ \ \ \ +\ A\ +\ B \ +\ C \ +o(1).
 \end{aligned}
\label{e3.15}
\end{equation}
Now,
\begin{equation}
 \begin{aligned}
  u(x,t,\epsilon) \tilde \sigma_{x}(x,t,\epsilon) &=(u_{0}+u_{1}H(-x+\phi(t),\epsilon) \\ 
  &\ \ \ \ \ \ +p(t)R(x-\phi(t),\epsilon))(-\sigma_{1} \frac{dH}{d\xi}(-x+\phi(t),\epsilon)) \\
  &= -u_{0}\sigma_{1} \frac{dH}{d \xi}(-x+\phi(t),\epsilon)-u_{1}\sigma_{1} H(-x+\phi(t),\epsilon)\frac{dH}{d \xi}(-x+\phi(t),\epsilon) \\
  &\ \ \ \ \ \ \ \ \ \ \ \ \ \ \ \ \ \ \ \ \ \ \ \ \ \ \ \ \ \ \ \ \ -\sigma_{1}p(t)\frac{dH}{d \xi}(-x+\phi(t),\epsilon)R(x-\phi(t),\epsilon) \\
  &=-(u_{0}\sigma_{1}+\frac{u_{1}\sigma_{1}}{2})\delta(x-\phi(t))+o_{\mathcal{D}^{\prime}}(1).
 \end{aligned}
\notag
\end{equation}
(Here $\frac{d}{d \xi}$ denotes differentiation with respect to the first variable.)\\
Hence 
\begin{equation}
 \begin{aligned}
 A =\int_{0}^{\infty} \int_{\mathbb{R}} u(x,t,\epsilon) \tilde \sigma_{x}(x,t,\epsilon) \theta(x,t) \ dx \ dt 
      &=-\int_{0}^{\infty} (u_{0}+\frac{u_{1}}{2})\sigma_{1} \theta(\phi(t),t)\ dt + o(1),\\
   &= \int_{0}^{\infty} \int_{\mathbb{R}} \hat u \tilde \sigma_{x}(x,t) \theta(x,t) \ dx \ dt+o(1).
 \end{aligned}
\notag
\end{equation}
The last step above follows from the fact that $\tilde \sigma$ is constant away from the curve $\Gamma$.\\
Again
\begin{equation}
 \begin{aligned}
  B &= \int_{0}^{\infty} \int_{\mathbb{R}} u(x,t,\epsilon) (e(x,t) \delta(x-\phi(t),\epsilon))_{x} \theta(x,t) \ dx \ dt \\
    &=\int_{0}^{\infty} \int_{\mathbb{R}} u_{0} (e(x,t) \delta(x-\phi(t),\epsilon))_{x} \theta(x,t) \ dx \ dt  \\
    & \ \ \ \ \ \ \ \ \ +\int_{0}^{\infty} \int_{\mathbb{R}} u_{1} H(-x+\phi(t),\epsilon) (e(x,t) \delta(x-\phi(t),\epsilon))_{x} \theta(x,t) \ dx \ dt \\
    &=-\int_{0}^{\infty} \int_{\mathbb{R}} u_{0} e(x,t) \delta(x-\phi(t),\epsilon) \theta_{x}(x,t) \ dx \ dt \\
    &  \ \ \ \ \ \ \ \ \ -\int_{0}^{\infty} \int_{\mathbb{R}} u_{1} H(-x+\phi(t),\epsilon) e(x,t) \delta(x-\phi(t),\epsilon) \theta_{x}(x,t) \ dx \ dt \\
    & \ \ \ \ \ \ \ \ \ \ +\int_{0}^{\infty} \int_{\mathbb{R}} u_{1} e(x,t) \delta(x-\phi(t),\epsilon) \frac{dH}{d \xi}(-x+\phi(t),\epsilon) \theta(x,t) \ dx \ dt \\
    &=-\int_{0}^{\infty} u_{0} e(t) \theta_{x}(\phi(t),t)\ dt - \int_{0}^{\infty} c u_{1} e(t) \theta_{x}(\phi(t),t) \ dt +\ 0 \ +o(1).
 \end{aligned}
\notag
\end{equation}
But $u_{0}+c u_{1}=u_{0}+(\frac{1}{2}-\frac{\sigma_{1}}{u_{1}^2})u_{1}= u_{0}+\frac{u_{1}}{2}-\frac{\sigma_{1}}{u_{1}}=\dot \phi(t)$.\\
Therefore $$B=-\int_{0}^{\infty} e(t) \dot \phi(t) \theta_{x}(\phi(t),t) \ dt+o(1). $$
Finally,
\begin{equation}
 \begin{aligned}
  C &= \int_{0}^{\infty} \int_{\mathbb{R}} (e(x,t) \delta(x-\phi(t),\epsilon))_{t} \theta(x,t) \ dx \ dt \\
    &= - \int_{0}^{\infty} \int_{\mathbb{R}} e(x,t) \delta(x-\phi(t),\epsilon) \theta_{t}(x,t) \ dx \ dt \\
    &= - \int_{0}^{\infty} e(t) \theta_{t}(\phi(t),t) dt + o(1).
 \end{aligned}
\notag
\end{equation}
Therefore, from \eqref{e3.15}, passing to the limit as $\epsilon \rightarrow 0$, we get 
$$0= \int_{0}^{\infty} \int_{\mathbb{R}} (\tilde \sigma_{t}(x,t) + \hat u \tilde \sigma_{x}(x,t)) \theta(x,t) \ dx \ dt - \int_{0}^{\infty} e(t) \frac{d \theta(\phi(t),t)}{dt} dt $$
which implies
\begin{equation}
\int_{0}^{\infty} \int_{\mathbb{R}} (\tilde \sigma_{t}(x,t)+\hat u \tilde \sigma_{x}(x,t)) \theta(x,t) \ dx \ dt -\int_{\Gamma} e(t) \frac{\partial \theta(\phi(t),t)}{\partial l} dl =0.
 \label{e3.16}
\end{equation}
From \eqref{e3.14} and \eqref{e3.16}, it follows that as $\epsilon \rightarrow 0$ the weak asymptotic solution \eqref{e3.9} indeed satisfies
the integral identities \eqref{e3.2} whereby the theorem follows.

\end{proof}

\begin{remark}
 Thus we have proved the existence of singular solutions for the system \eqref{e1.1} with arbitrary Riemann type initial data \eqref{e3.3} with
singular concentration in the \textit{stress} variable $\sigma$. This solution, as discussed in the introduction, seems to be the physically
relevant one. 
\end{remark}

\section{A few concluding remarks}
It would be important to note that in the previous section on the existence of generalised delta-shock wave type solutions, we didn't 
consider any entropy conditions. There are mainly a couple of reasons behind this. Firstly, if we go through the proofs regarding the existence
of weak asymptotic solutions or generalised delta-shock wave type solutions, it would be interesting to note that the proofs hold true 
for arbitrary Riemann type initial data and hence further restrictions imposed by entropy conditions are not required. Secondly, since we
do not really consider the related problem of uniqueness of solutions in this article, the entropy conditions could be avoided. \\ 
\\
We would like to remark that the \textit{overcompressivity condition} for the delta-shock wave type solutions for the system \eqref{e1.1} with Riemann type initial data \eqref{e3.3} (see for instance \cite{a1,k1})
imposes the condition 
\begin{equation}
u_{0}<\dot \phi{(t)}<u_{0}+u_{1},
 \notag
\end{equation}
which on simplification gives $u_{1}>0$ and $-\frac{u_{1}}{2}<\frac{\sigma_{1}}{u_{1}}<\frac{u_{1}}{2}$.

\end{document}